\documentclass[final,12pt]{elsarticle}

\makeatletter
\def\ps@pprintTitle{%
 \let\@oddhead\@empty
 \let\@evenhead\@empty
 \def\@oddfoot{\centerline{\thepage}}%
 \let\@evenfoot\@oddfoot}
\makeatother

\usepackage{graphicx}
\usepackage{amssymb}
\usepackage{amsthm}
\usepackage{amsmath}
\usepackage{mathtools}
\usepackage{lineno}
\usepackage{xcolor}
\usepackage{tikz}
\usepackage{cleveref}
\usepackage{subcaption}
\usepackage{wrapfig}
\usepackage{todonotes}
\usepackage{url}

\newtheorem{theorem}{Theorem}
\newtheorem{proposition}{Proposition}
\newtheorem{corollary}{Corollary}
\newtheorem{lemma}{Lemma}
\newtheorem{remark}{Remark}
\newtheorem{problem}{Problem}
\theoremstyle{remark}
\newtheorem*{example}{Example}
\newtheorem*{acknowledgements}{Acknowledgements}

\newcommand{\bolda}{\mathbf{a}}

\newcommand{\boldc}{\mathbf{c}}
\newcommand{\bolde}{\mathbf{e}}
\newcommand{\boldf}{\mathbf{f}}
\newcommand{\boldh}{\mathbf{h}}
\newcommand{\boldv}{\mathbf{v}}
\newcommand{\boldx}{\mathbf{x}}
\newcommand{\boldy}{\mathbf{y}}

\newcommand{\boldzero}{\mathbf{0}}
\newcommand{\boldalpha}{\boldsymbol{\alpha}}

\newcommand{\boldomega}{\boldsymbol{\omega}}
\newcommand{\imag}{\mathbf{i}}

\newcommand{\C}{\mathbb{C}}

\newcommand{\R}{\mathbb{R}}
\newcommand{\Z}{\mathbb{Z}}

\DeclareMathOperator{\conv}{conv}

\DeclareMathOperator{\init}{init}
\DeclareMathOperator{\newt}{Newt}

\newcommand{\inner}[2]{\langle \,#1\,,\, #2\, \rangle}

\begin{document}

\begin{frontmatter}


\title{On the root count of algebraic Kuramoto equations in cycle networks with uniform coupling}
\author{Tianran Chen}
\ead{ti@nranchen.org}
\author{Evgeniia Korchevskaia}
\ead{ekorchev@aum.edu}
\address{Department of Mathematics, Auburn University Montgomery, Montgomery, Alabama}
\fntext[fn1]{
    TC and EK are supported in part by NSF research grant DMS-1923099 
    and a grant from the Auburn University at Montgomery (AUM) Grant-in-Aid Program.
    EK is also supported by the Undergraduate Research fund of AUM Department of Mathematics.}




\begin{abstract}
    The Kuramoto model is a classical model used in the study of
    spontaneous synchronizations in networks of coupled oscillators.
    In this model, frequency synchronization configurations can be formulated
    as complex solutions to a system of algebraic equations.
    Recently, upper bounds to the number of frequency synchronization configurations 
    in cycle networks of $N$ oscillators  were calculated 
    under the assumption of generic non-uniform coupling.
    In this paper, we refine these results for the special cases of uniform coupling.
    In particular, we show that when, and only when, $N$ is divisible by 4,
    the upper bound for the number of synchronization configurations
    in the uniform coupling cases is significantly less than the bound in the non-uniform coupling cases.
    This result also establishes an explicit formula for the gap between
    the birationally invariant intersection index and the 
    Bernshtein-Kushnirenko-Khovanskii bound
    for the underlying algebraic equations.
\end{abstract}


\end{frontmatter}


\section{Introduction}

The spontaneous synchronization of oscillators is a ubiquitous phenomenon
that appear naturally in many seemingly independent biological, mechanical, and electrical systems.
One of the classical models in the study of synchronization is the Kuramoto model \cite{Kuramoto1975Self}, 
which describes the dynamics of a network of coupled oscillators. 
The Kuramoto model has been extensively studied in the recent decades
and remains an enduring subject for the modeling of synchronization phenomena 
arising from the areas of science and engineering. 
Originally, the Kuramoto model had been applied to infinite complete networks (with all-to-all coupling). 
In order to adapt the model to complex topologies, 
numerous reformulations of the Kuramoto model have been introduced and studied, 
both analytically and numerically. 
See, e.g., review articles \cite{ARENAS200893, dorfler_synchronization_2014, RODRIGUES20161}. 
One of such generalizations of the Kuramoto model investigates synchronizations within cycle networks
(i.e., ring-like networks)
\cite{DelabaysColettaJacquod2016Multistability, Denes2019Pattern, Ha2012Basin, ManikTimmeWitthaut2017Cycle, Rogge2004Stability, Roy2012Synchronized, Xi2017Synchronization}. 
From an algebraic view point, 
the upper bound on the number of frequency synchronization configurations
(including complex configurations)
that can exist in a cycle network of $N$ oscillators with generic non-uniform coupling
is shown to be $N \binom{N-1}{ \lfloor (N-1)/2 \rfloor }$~\cite{ChenDavisMehta2018Counting,dal2019faces}.

The main contribution of this paper is a significant refinement 
of this upper bound for the case of cycle networks with \emph{uniform coupling}.
We show that if $N$ is not divisible by 4,
despite being a very special case, 
the generic number of complex synchronization configurations
is the same as in the cases with generic non-uniform coupling.
On the other hand, if $N$ is divisible by 4,
then the generic synchronization configuration count is significantly lower 
than the count in the case of non-uniform coupling.
This result also quantifies the gap between the 
birationally invariant intersection index of a family of rational functions 
over the toric variety $(\C^*)^n$
and the Bernshtein-Kushnirenko-Khovanskii bound \cite{Bernshtein1975Number, Khovanskii1978Newton, Kushnirenko1976Polyedres}
of a generic algebraic system in this family.


This paper is organized as follows. 
In \Cref{sec:Review and problem statement}, 
we review the Kuramoto equations and their algebraic formulation  
and state the root counting problems which this paper focuses on. 
\Cref{sec: Adjacency polytope} describes the construction of the adjacency polytope and 
explores its geometric properties which are central to our main arguments. 
Then in \Cref{sec: counting roots}, we establish the main result which is the 
generic root count for Kuramoto equations arising from cycle networks with uniform coupling. 
The conclusion follows in \Cref{sec: conclusion}.

\begin{acknowledgements}

This project grew out of authors' discussion with
Anton Leykin and Josephine Yu in 2017 and with Anders Jensen and Yue Ren in 2018
on closely related problems.
The authors also thank Rob Davis for kindly sharing his insights into
the geometric structures of adjacency polytopes.
\end{acknowledgements}

\section{Preliminaries and problem statements}
\label{sec:Review and problem statement}
\subsection{Kuramoto model with uniform coupling and synchronization equations}
\label{sec: Kuramoto model}
\begin{wrapfigure}[9]{r}{0.25\textwidth}
    \centering
    \begin{tikzpicture}[
            every node/.style={circle,thick,draw},
            every edge/.style={draw,thick}]
            \node (0) at ( 0,1.8) {0};
            \node (1) at (-0.9,0.9) {1};
            \node (2) at (0,0) {2};
            \node (3) at (0.9,0.9) {3};
            \path (1) edge (0);
            \path (1) edge (2);
            \path (3) edge (2);
            \path (3) edge (0);
        \end{tikzpicture}
        \caption{A cycle network of 4 oscillators}
        \label{fig: C4 network}
\end{wrapfigure}
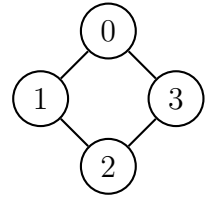 
The network of $N$ coupled nonlinear oscillators can be modeled by an undirected graph $G$
with vertices $V(G)$ and edges $E(G)$ representing 
the oscillators and their connections respectively.
In addition, each oscillator $i$ has its natural frequency $\omega_i$,
while a nonzero constant $K$ quantifies the coupling strength between two oscillators.
In this paper, we focus exclusively on the cases of cycle graphs of $N$ vertices.
That is, we only consider $G = C_N$
with $V(C_N) = \{0,1,\dots,n\}$, $n=N-1$ and
$E(C_N) = \{ \{0,1\}, \{1,2\},\dots,\{i,i+1\},\dots,\{n-1,n\},\{n,0\} \}$.
For example, \Cref{fig: C4 network} shows a cycle network of 4 oscillators.
We also assume that the oscillators are non-homogeneous,
i.e., $\omega_0,\dots,\omega_n$ are distinct,
but the coupling is uniform, i.e., the strength of the coupling along any edge is $K$.
Under these assumptions, 
the dynamics of the system is described by the Kuramoto model
with the governing equations
\begin{equation}\label{equ:kuramoto-sin}
    \frac{d \theta_i}{dt}=\omega_i - K \sum_{j \in \mathcal{N}_{C_N}(i)} \sin(\theta_i - \theta_j) 
    \quad\text{for } i = 0,\dots,n,
\end{equation}
where 
$\theta_i$ is the phase angle of the $i$-th oscillator 
and $\mathcal{N}_{C_N}(i)$ is the set of adjacent nodes of node $i$ in the $C_N$.
Frequency synchronization configurations (simply synchronization configurations, hereafter) 
are defined to be configurations of $(\theta_0,\dots,\theta_n)$ 
for which all oscillators are tuned to have the exact same angular velocity.
That is, there is a single constant $c$ such that $\frac{d\theta_i}{dt} = c$ for $i=0,\dots,n$.
By adopting a proper rotational frame of reference,
we can further assume $\theta_0 = 0$ and $c = 0$. 
Then the synchronization configurations are defined by the equilibrium conditions 
$\frac{d \theta_i}{dt}=0$ for $i=1,\dots,n$.
That is, they are solutions to the system of transcendental equations
\begin{equation}\label{equ:kuramoto-sin=0}
    0 = 
    \omega_i - K \sum_{j \in \mathcal{N}_{C_N}(i)} \sin(\theta_i - \theta_j) 
    \quad\text{for } i = 1,\dots,n.
\end{equation}

Throughout this paper, we will use the following notations. Let
$f = \sum_{\bolda \in S} c_{\boldalpha} \boldx ^ \bolda$
denote a Laurent polynomial in $n$ variables $\boldx = (x_1,\dots,x_n)$ with coefficients in $\C$,
where the finite set $S \subset \Z^n$, known as its support, collects the exponents
and $\boldx^\bolda = x_1^{a_1} \cdots x_n^{a_n}$ represents the monomial
with exponent $\bolda = (a_1,\dots,a_n)^\top$.
The Newton polytope of $f$ is the set $\newt(f):=\conv(S) \subset \R^n$. 
With respect to a nonzero vector $\boldalpha \in \R^n$, the initial form of $f$ is
$\init_{\boldalpha}f := \sum_{\bolda \in (S)_\boldv} c_{\boldalpha} \boldx ^ \bolda$, 
where $(S)_\boldv$ is the subset of $S$ on which the linear functional $\inner{ \boldv }{ \cdot }$ is minimized over $S$. 
For a system $\boldf = (f_1,\dots,f_n)^\top$ of Laurent polynomials,
the initial system with respect to $\boldalpha \in \R^n$ is
$\init_{\boldalpha} \boldf := (\init_{\boldalpha}f_1, \dots , \init_{\boldalpha} f_n)^\top$. 
While considering the root count of the system $\boldf$, 
we will make use of the Bernshtein's theorem \cite{Bernshtein1975Number}, 
which states that the generic root count of the system $f$ in $(\C^*)^n = (\C \setminus \{0\})^n$
is given by the mixed volume of the Newton polytopes $\newt(f_i)$, $i=1,\dots, n$. 
This count is also known as the Bernshtein-Kushnirenko-Khovanskii (BKK) bound
\cite{Bernshtein1975Number,Kushnirenko1976Polyedres,Khovanskii1978Newton}.

The central question of finding the maximum number of synchronization configurations for a cycle network of oscillators with uniform coupling is equivalent to the root counting question of the system ~\eqref{equ:kuramoto-sin=0}. To leverage the power of root counting results from algebraic geometry,
the transcendental equations~\eqref{equ:kuramoto-sin=0}
can be reformulated into an algebraic system
via the change of variables $x_i = e^{\imag \theta_i}$
for $i = 0,\dots,n$ 
where $\imag = \sqrt{-1}$ and $x_0 = e^{\imag 0} = 1$ corresponds to the fixed phase angle
of the reference oscillator. Then 
$\sin(\theta_i - \theta_j) =  \frac{1}{2 \imag}(\frac{x_i}{x_j} - \frac{x_j}{x_i})$,
and~\eqref{equ:kuramoto-sin=0} is transformed into a system of 
$n$ Laurent polynomial equations 
$\mathbf{f} = (f_1,\dots,f_n)^\top = \boldzero$ 
in the $n$ complex variables $\boldx = (x_1,\dots,x_n)$ given by
\begin{equation} 
    \label{equ:kuramoto-alg}
    f_{i}(x_1,\dots,x_n) = 
    \omega_i - a \sum_{j \in \mathcal{N}_{C_N}(i)} 
        \left(
            \frac{x_i}{x_j} - \frac{x_j}{x_i}
        \right)
    = 0
    \quad \text{for } i = 1,\dots,n,
\end{equation} 
where $a = \frac{K}{2\imag}$.
This system captures all synchronization configurations in a way that 
the real solutions to~\eqref{equ:kuramoto-sin=0} correspond to the
complex solutions of~\eqref{equ:kuramoto-alg} with each $|x_i| = |e^{\imag \theta}| = 1$,
i.e., solutions on the real torus $(S^1)^n$.

\subsection{Problem statements}

Counting solutions of an algebraic system on real torus $(S^1)^n$
is a notoriously difficult problem.
Using Morse inequalities and the Betti numbers of the real torus,
lower bounds on the generic solution count was established by
Baillieul and Byrnes in certain cases~\cite{Baillieul1982}
(e.g. the cases of homogeneous oscillators with nondegenerate synchronization states).
An upper bound for the number of solutions is also established 
to be ${2N-2}\choose{N-1}$ in the same papers 
by bounding the total number of complex solutions to~\eqref{equ:kuramoto-alg}.
This upper bound does not take into consideration the graph topology and coupling coefficients
and only depends on the number of oscillators $N$.
Following this approach of counting complex solutions,
later works (e.g. root counting results given by Guo and Salam~\cite{Guo1990} 
and Molzahn, Mehta, and Niemerg~\cite{MolzahnMehtaNiemerg2016Toward})
suggest that sparse networks tend to have less synchronization configurations. 
These observations have motivated a study on the tighter upper bound 
for the number of synchronization configurations in cycle networks 
with non-uniform coupling coefficients~\cite{Chen2019Unmixing, ChenDavisMehta2018Counting}
where a sharp upper bound (counting complex synchronization configurations)
is shown to be $N \binom{N-1}{ \lfloor (N-1)/2 \rfloor }$.
In this paper we provide a significant refinement of the 
root count for the cases of cycle networks with uniform coupling.


An important result from the intersection theory is that, over the field of complex numbers,
the maximal behavior is also, in a sense, the generic behavior
(e.g. the Theorem of Bertini~\cite{SommeseWampler2005Numerical}).
The key question we aim to answer is therefore the following generic root count question.

\begin{problem}\label{prb:one}
    Given a cycle graph $C_N$ of $N$ nodes
    and generic choices of parameters $\boldomega = (\omega_1,\dots,\omega_n)$ and $a$,
    what is the total number of isolated complex solutions
    to the algebraic Kuramoto system~\eqref{equ:kuramoto-alg}?
\end{problem}

This question can also be stated in terms of birationally invariant intersection index.
Since the $i$-th Laurent polynomial in~\eqref{equ:kuramoto-alg} is a linear combination of
$1$ and $\ell_i := \sum_{j \in \mathcal{N}_{C_N}(i)} ( x_i / x_j - x_j / x_i)$
with generic coefficient.
It can be considered as a generic element in the vector space of rational functions spanned by
$\{ 1, \ell_i \}$.
We are therefore interested in the intersection index of $n$ generic elements
from these vector spaces in $(\C^*)^n$.
This is precisely the birationally invariant intersection index
\cite{KavehKhovanskii2010Mixed}.

\begin{problem}\label{prb:two}
    Given a cycle graph $C_N$ of $N$ nodes, let
    \[
        L_i = \operatorname{span} 
        \left\{ 
            1, 
            \sum_{j \in \mathcal{N}_{C_N}(i)} \left( \frac{x_i}{x_j} - \frac{x_j}{x_i} \right)
        \right\}
    \]
    be the $\C$-vector space spanned by two rational functions
    for each $i=1,\dots,n$.
    What is the intersection index $[\, L_1,\dots,L_n \,]$ ?
\end{problem}

This intersection index is less than or equal to the 
BKK bound
for the same set of equations.
In this paper, we show that there is a gap between the intersection index described above
and the BKK bound if and only if $N$ is divisible by 4.
Indeed, in this case, $[L_1,\dots,L_n]$ is significantly smaller than the BKK bound
in the sense that the ratio of the two goes to zero as $N \to \infty$.

\section{Adjacency polytope} \label{sec: Adjacency polytope}
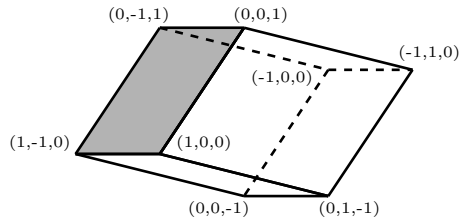
\begin{wrapfigure}[10]{r}{0.45\textwidth}
    \centering
    \begin{tikzpicture} [scale=0.28]
    \draw [line width=1pt ] (0,4)-- (8,2);
    \draw [line width=1pt] (8,2)-- (4,-4);
    \draw [line width=1pt] (4,-4)-- (-4,-2);
    \draw [line width=1pt] (-4,-2)-- (0,4);
    \draw [line width=1pt, dashed ] (-4,4)-- (4,2);
    \draw [line width=1pt, dashed] (4,2)-- (8,2);
    \draw [line width=1pt, dashed] (4,2)-- (0,-4);
    \draw [line width=1pt] (-8,-2)-- (-4,-2);
    \draw [line width=1pt] (-4,-2)-- (0,4);
    \draw [line width=1pt] (0,4)-- (-4,4);
    \draw [line width=1pt] (-4,4)-- (-8,-2);
    \draw [line width=1pt] (0,4)-- (-4,-2);
    \draw [line width=1pt] (-8,-2)-- (0,-4);
    \draw [line width=1pt] (0,-4)-- (4,-4);
    \draw [line width=1pt] (4,-4)-- (-4,-2);
    \draw [line width=1pt] (-4,-2)-- (-8,-2);
    \begin{tiny}
    \draw[color=black] (-5,4.6) node {(0,-1,1)};
    \draw[color=black] (0.8,4.6) node {(0,0,1)};
    \draw[color=black] (-9.7,-1.5) node {(1,-1,0)};
    \draw[color=black] (-1.9,-1.5) node {(1,0,0)};
    \draw[color=black] (-1.2,-4.6) node {(0,0,-1)};
    \draw[color=black] (5,-4.6) node {(0,1,-1)};
    \draw[color=black] (8.8,2.8) node {(-1,1,0)};
    \draw[color=black] (1.8, 1.6) node {(-1,0,0)};
    \end{tiny}
    \fill[fill=black,fill opacity=0.3] (-8,-2) -- (-4,4) -- (0,4) -- (-4,-2) -- cycle;
    \end{tikzpicture}
    \caption{Adjacency polytope for a cycle network of 4 oscillators}
    \label{fig: adj polytope}
\end{wrapfigure}

Recent studies suggest that the range of possible synchronization configurations
is strongly tied to the network topology \cite{Bronski2016Graph, MolzahnMehtaNiemerg2016Toward}.
A promising approach to elucidate this connection \cite{Chen2019Directed, ChenDavisMehta2018Counting,Chen2019Unmixing}
makes use of a construction known as an ``adjacency polytope''. 
This method allows us to encode the network topology and 
provides valuable insights into the algebraic structure of the Kuramoto equations. 
The adjacency polytope constructed in this context coincides with
the symmetric edge polytope introduced earlier in the study of the roots of Ehrhart polynomials~\cite{Matsui2011Roots}.
The geometric structure of adjacency polytopes has been instrumental in the study of 
generic root count of the algebraic Kuramoto equations~\eqref{equ:kuramoto-alg} 
in the case of cycle networks with generic non-uniform coupling
\cite{ChenDavisMehta2018Counting}.
In this paper, we extract more refined initial form information 
from the polytopes and establish a sharper bound on the generic root count
in the case of uniform coupling.

For the cycle graph $C_N$, its \emph{adjacency polytope} is defined as
\begin{equation}\label{equ:AP}
    \nabla_{C_N} = \conv \left\{
        \pm (\bolde_i - \bolde_j)
    \right\}_{\{i,j\} \in E(C_N)},
\end{equation}
where $\bolde_i \in \R^n$ is the column vector with 1 on the $i$-th position
and zero elsewhere for $i=1,\dots,n$, and $\bolde_0 = \boldzero$. The polytope $\nabla_{C_N}$ is a full-dimensional centrally symmetric lattice polytope. 
It is originally constructed as the Newton polytope of the randomized system 
$\boldf^{R} := R \cdot \boldf$,
which is created from a nonsingular $n \times n$ matrix $R = [ r_{ij} ]$ with generic entries.
Here, components of $\boldf^R = (f_1^R,\dots,f_n^R)$ are of the form
\begin{equation}
    f_{k}^R = 
    c_k -
    \sum_{\{i,j\} \in E(C_N),i<j} a_{ijk}^R 
    \left(
        \frac{x_i}{x_j} - \frac{x_j}{x_i}
    \right)
    \quad\text{for } k = 1,\dots,n,
\end{equation}
where $\boldc = (c_1,\dots,c_n)^\top = R \, \boldomega$, 
$E(C_N)$ is the edge set 
of $C_N$, and $a^R_{ijk} = a (r_{ki} - r_{kj})$.
Each component in $\boldf^R$ has the same set of terms,
and the (unmixed) Newton polytope of $\boldf^R$,
$\newt(\boldf^R)$,
is precisely the adjacency polytope $\nabla_{C_N}$.
Since $R$ is nonsingular, $\boldf$ and $\boldf^R = R \cdot \boldf$
share the same zero set.
Yet, the randomization simplifies the algebraic structure of the problem 
and allows us to utilize the known results concerning the triangulation of 
$\nabla_{C_N}$~\cite{ChenDavis2018Toric, ChenDavisMehta2018Counting}.
The normalized volume of $\nabla_{C_N}$, 
also known as \emph{adjacency polytope bound}, 
is the BKK bound for both the algebraic Kuramoto system~\eqref{equ:kuramoto-alg}
and the randomized system $\boldf^R$.

In the following discussion, 
an important role is given to the facets ($(n-1)$-dimensional faces) of $\nabla_{C_N}$. 
As described in Ref.~\cite{ChenDavisMehta2018Counting}, 
when $N$ is odd, $\nabla_{C_N}$ is unimodularly equivalent to the 
\emph{del Pezzo polytope} \cite{Nill2005Classification}.  
The number of facets of $\nabla_{C_N}$ is $N {{N-1}\choose{(N-1)/2}}$.
Each facet is simplicial, unimodular, and given by
\begin{equation}
\begin{split}
\label{def_oddfacet}
    \conv \Biggl\{
        \lambda_j (\bolde_i - \bolde_j)
        \,\Bigm|\, 
       \{i,j\} \in E(C_N) \setminus \bigl\{\{p,q\}\bigr\}  \text{ for some } \{p,q\} \in E(C_N) , \\   \lambda_1,\dots, \lambda_{q-1}, \lambda_{q+1}, \dots, \lambda_{n+1} \in \{ \pm 1 \},\; 
       \sum_{j=1}^{q-1} \lambda_j +\sum_{j=q+1}^{n+1} \lambda_j =0
    \Biggr\}.
\end{split}
\end{equation}

When $N$ is even, $\nabla_{C_N}$ has $N\choose{N/2}$ facets \cite{ChenDavis2018Toric, dal2019faces}. 
Each facet $F$ is defined by exactly $N = n + 1$ vertices 
and is of the form
\begin{equation} \label{def_evenfacet}
    \conv \Biggl\{
        \lambda_j (\bolde_i - \bolde_j)
        \, \Bigm| \, \{i,j\} \in E(C_N), \,
        \lambda_1,\dots \lambda_{n+1} \in \{ \pm 1 \},\; 
        \sum_{j=1}^{n+1} \lambda_j =0
    \Biggr\}.
\end{equation}     

\begin{example} \label{example-c4} (Running example, 4-cycle). 
    Our reference example throughout this paper will be a a cycle network with $N=4$ coupled oscillators. See Figure \ref{fig: C4 network}. Synchronization configurations of this network are characterized by the Algebraic Kuramoto equations 
    \begin{align} \label{ex: C-4 synchr.system}
        \omega_1 - a\left(\frac{x_1}{x_0}-\frac{x_0}{x_1}+\frac{x_1}{x_2}-\frac{x_2}{x_1}\right)=0, \nonumber\\
        \omega_2 - a\left(\frac{x_2}{x_1}-\frac{x_1}{x_2}+\frac{x_2}{x_3}-\frac{x_3}{x_2}\right)=0,\\
        \omega_3 - a\left(\frac{x_3}{x_2}-\frac{x_2}{x_3}+\frac{x_3}{x_0}-\frac{x_0}{x_3}\right)=0. \nonumber
    \end{align}
The adjacency polytope associated with this network 
is a parallelepiped as illustrated on Figure \ref{fig: adj polytope}.
Its six facets are
\begin{align} \label{facets 4-cycle}
    &\conv \left\{
    \left[ \begin{smallmatrix} 1\\0\\0 \end{smallmatrix}
    \right],
    \left[\begin{smallmatrix} \phantom{-}1\\-1\nonumber\\\phantom{-}0 \end{smallmatrix}\right], 
    \left[\begin{smallmatrix} \phantom{-}0\\{-}1\\\phantom{-}1 \end{smallmatrix} \right],
    \left[\begin{smallmatrix} 0\\0\\1 \end{smallmatrix}
    \right] \right\},
    &&\conv \left\{
    \left[\begin{smallmatrix} 1\\0\\0 \end{smallmatrix}\right],
    \left[\begin{smallmatrix} {-}1\\\phantom{-}1\\\phantom{-}0 \end{smallmatrix}\right], 
    \left[\begin{smallmatrix} \phantom{-}0\\\phantom{-}1\\{-}1 \end{smallmatrix}\right],
    \left[\begin{smallmatrix} 0\\0\\1 \end{smallmatrix}\right]
    \right\},\\ 
    &\conv \left\{
    \left[\begin{smallmatrix} 1\\0\\0 \end{smallmatrix}\right],
    \left[\begin{smallmatrix} \phantom{-}1\\{-}1\\\phantom{-}0 \end{smallmatrix}\right], 
    \left[\begin{smallmatrix} \phantom{-}0\\\phantom{-}1\\{-}1 \end{smallmatrix}\right],
    \left[\begin{smallmatrix} \phantom{-}0\\\phantom{-}0\\{-}1 \end{smallmatrix}\right]
    \right\},
    &&\conv \left\{
    \left[\begin{smallmatrix} {-}1\\\phantom{-}0\\\phantom{-}0 \end{smallmatrix}\right],
    \left[\begin{smallmatrix} {-}1\\\phantom{-}1\\\phantom{-}0 \end{smallmatrix}\right], 
    \left[\begin{smallmatrix} \phantom{-}0\\{-}1\\\phantom{-}1 \end{smallmatrix}\right],
    \left[\begin{smallmatrix} 0\\0\\1 \end{smallmatrix}\right]
    \right\},\\
    &\conv \left\{
    \left[\begin{smallmatrix} {-}1\\\phantom{-}0\\\phantom{-}0 \end{smallmatrix}\right],
    \left[\begin{smallmatrix} {-}1\\\phantom{-}1\\\phantom{-}0 \end{smallmatrix}\right], 
    \left[\begin{smallmatrix} \phantom{-}0\\\phantom{-}1\\{-}1 \end{smallmatrix}\right],
    \left[\begin{smallmatrix} \phantom{-}0\\\phantom{-}0\\{-}1 \end{smallmatrix}\right]
    \right\},
    &&\conv \left\{
    \left[\begin{smallmatrix} {-}1\\\phantom{-}0\\\phantom{-}0 \end{smallmatrix}\right],
    \left[\begin{smallmatrix} \phantom{-}1\\{-}1\\\phantom{-}0 \end{smallmatrix}\right], 
    \left[\begin{smallmatrix} \phantom{-}0\\{-}1\\\phantom{-}1 \end{smallmatrix}\right],
    \left[\begin{smallmatrix} \phantom{-}0\\\phantom{-}0\\{-}1 \end{smallmatrix}\right]
    \right\}.\nonumber
\end{align}        
The normalized volume of each facet is 2,
and the adjacency polytope bound is 12.
This also agrees with the BKK bound of \eqref{ex: C-4 synchr.system}.
We will show, however, that
the generic root count under the uniform coupling condition,
i.e., the intersection index $[L_1,L_2,L_3]$, is only 6.

\end{example}

We next prove some important properties of the facets of $\nabla_G$. For a facet $F$ of $\nabla_G$, we define its \emph{facet matrix} to be a matrix $V$ 
whose columns are the vertices of $F$, arranged in such a way, 
that vertex $\lambda_k (\bolde_i - \bolde_j)$ 
corresponds to the $i$-th column of $V$. 
Let us denote by $V^*$ the reduced row echelon form of $V$. 
We will refer to $V^*$ as a \emph{reduced facet matrix}. 
Finally, we define a \emph{facet reduction matrix} $Q$ as an $(N-1) \times (N-1)$ matrix that satisfies $QV=V^*$. 

\begin{lemma} \label{Lemma: reduction matrix Q}
    For a facet $F$ of $\nabla_G$ with its facet matrix $V$ and reduced facet matrix $V^*$,
    the facet reduction matrix $Q$ is a unimodular integer matrix.
\end{lemma}

\begin{proof}
    If $N$ is odd, then $V$ is a square unimodular matrix with integer entries.
    Its reduced row echelon form is therefore the identity matrix,
    and $Q = V^{-1}$ is a unimodular integer matrix.
    
    If $N$ is even,
    employing the description of the facets given in \eqref{def_evenfacet},
    we write each facet matrix $V$ as 
    \begin{equation}
        V=\begin{bmatrix}
            -1 &  1     &        &    \\
               & \ddots & \ddots &    \\
               &        & -1     & 1  \\
         \end{bmatrix}
         \begin{bmatrix}
           \lambda_1 &        &              \\
                     & \ddots &              \\
                     &        & \lambda_N    \\
         \end{bmatrix}.
    \end{equation} 
    We define an $(N-1) \times (N-1)$ unimodular (integer) matrix 
    \begin{equation}
        Q \coloneqq  
        \begin{bmatrix}
           -\lambda_{1} &               &           &                        \\
                        &  -\lambda_{2} &           &                        \\  
                        &               & \ddots    &                          \\
                        &               &           & -\lambda_{N-1}             \\
        \end{bmatrix}
        \begin{bmatrix}
         1  &  1 &  \dots  &   1               \\
            &  1 &  \dots  &   1          \\
            &    &  \ddots &   \vdots        \\
            &    &         &   1         \\
        \end{bmatrix},
    \end{equation}
    which has determinant $\pm 1$ since $\lambda_k = \pm 1$.
    Via a direct computation, we observe that 
    \begin{equation} \label{V*}
        QV = \begin{bmatrix}
             1 &    &        &        &  -\lambda_{1} \lambda_{N}  \\
               &  1 &        &        &  -\lambda_{2} \lambda_{N} \\
           
               &    & \ddots &        &   \vdots \\
               &    &        &  1     & -\lambda_{N-1} \lambda_{N}  \\
        \end{bmatrix}=V^*.
    \end{equation}
    Therefore, the integer matrix $Q$ is the facet reduction matrix.
\end{proof}

It has been mentioned in Ref.~\cite{ChenDavis2018Toric} that
facets of $\nabla_G$ are unimodularly equivalent to one another. 
In \Cref{prop.unimod.equiv}, we provide an equivalent proof in terms of facet matrices. 

\begin{proposition} \label{prop.unimod.equiv}
    Let $F$ and $F'$ be two facets of the adjacency polytope $\nabla_{C_N}$, 
    and let $V$ and $V'$ be their corresponding facet matrices. 
    Then there exist a unimodular $(N-1) \times (N-1)$ matrix $U$ 
    and a $N \times N$ permutation matrix $P$ such that $UVP=V'$.
\end{proposition}

\begin{proof}
\noindent(Odd N) 
For odd $N$, $V$ is square and unimodular, 
and therefore $V^{-1}$ is an integer matrix. 
We let $U = V' V^{-1}$, 
and let $P=I_{N-1}$ be the $(N-1) \times (N-1)$ identity matrix, 
then $UVP=V'$.

(Even $N$) For even $N$,  let 
$(\lambda_{1}, \dots \lambda_{N}), (\lambda_1',\dots,\lambda_N') \in \{-1,1\}^N$ 
be the indices of the $F$ and $F'$ respectively.
These indices satisfy the conditions 
$\sum_{i=1}^N \lambda_i  = 0$ and
$\sum_{i=1}^N \lambda_i' = 0$.
Therefore exactly half of the entries in each collection of indices are positive.
Let $Q$ and $Q'$ be the facet reduction matrices of $F$ and $F'$ respectively.
As stated in \eqref{V*}, the corresponding reduced facet matrices are given by
\begin{align*} 
    V^* = QV &= 
    \left[
    \begin{smallmatrix}
         1 &    &        &        &  -\lambda_{1} \lambda_{N}  \\
           &  1 &        &        &  -\lambda_{2} \lambda_{N} \\
           &    & \ddots &        &   \vdots \\
           &    &        &  1     & -\lambda_{N-1} \lambda_{N}  \\
    \end{smallmatrix}
    \right]
    &&\text{and}&
    (V')^* = Q'V' &= 
    \left[
    \begin{smallmatrix}
         1 &    &        &        &  -\lambda_{1}' \lambda_{N}'  \\
           &  1 &        &        &  -\lambda_{2}' \lambda_{N}' \\
           &    & \ddots &        &   \vdots \\
           &    &        &  1     & -\lambda_{N-1}' \lambda_{N}'  \\
    \end{smallmatrix}
    \right].
\end{align*}
We observe that if $\lambda_N$ is positive, then the list $\lambda_1, \dots, \lambda_{N-1}$ contains $\frac{N}{2}-1$ positive and $\frac{N}{2}$ negative entries. If $\lambda_N$ is negative, then the list $\lambda_1. \dots ,\lambda_{N-1}$ contains $\frac{N}{2}-1$ negative and $\frac{N}{2}$ positive entries. 
It follows that the last columns of both $V^*$ and $(V')^*$ 
have exactly $\frac{N}{2}$ entries equal $1$ and $\frac{N}{2}-1$ entries equal $-1$. 
Therefore $V^*$ and $(V')^*$ 
are equal up to a permutation of the entries in the last column. 
In other words, there exist permutations matrices $L$ and $P$
of sizes $(N-1) \times (N-1)$ and $N \times N$ respectively 
such that 
\begin{equation}
    L V^* P=(V')^*.
\end{equation}
Let $U = (Q')^{-1} L Q$, which is unimodular since $Q'$, $L$, and $Q$ are all unimodular, 
then we have $U V P = V'$, as desired. 
\end{proof}

\subsection{Face and facet subsystems}

Faces of the adjacensy polytope $\nabla_{C_N}$ 
give rise to face and facet subsystems
that form the foundation of our root counting argument.
Let $F$ be a positive-dimensional face of $\nabla_{C_N}$, 
then the \emph{face subsystem} induced by $F$ is the system given by
\begin{equation}
    f_{F,k}^R =
    c_k - 
    \sum_{(\bolde_i - \bolde_j) \in F} (a_{ijk} - a_{jik}) \, \frac{x_i}{x_j}
    \quad\text{ for } k=1,\dots,n,
\end{equation}
which consists of all the terms in the algebraic Kuramoto system
with exponents vectors in $F$ together with the constant terms.
Using the compact vector exponent notation, 
we can write the system as
\begin{equation}
    \boldf_F^R =
    \boldc - a \sum_{(\bolde_i - \bolde_j) \in F} 
    (R_i - R_j) \, (\boldx^{(\bolde_i - \bolde_j)})^\top
    \label{eq: def. facet subsystem}
\end{equation}
where $R_k$ is the $k$-th column of $R$ if $k \ne 0$ 
and the zero vector otherwise.
If $F$ is a facet, we call $\boldf_F^R$ a \emph{facet subsystem}. Facet subsystems correspond to facet subnetworks investigated in Ref.~\cite{Chen2019Directed}.

\section{Counting roots}
\label{sec: counting roots}
The main goal here is to provide a refined generic root count
in the special case of uniform coupling,
which will be the answer for both \Cref{prb:one} and \Cref{prb:two}. This will be done via analysis of the roots of the much simpler face subsystems described above. Throughout this section, we assume the choices of the natural frequencies $\omega_i$'s
and the complex coupling coefficient $a$ to be generic. 
This can be interpreted in terms of Zariski topology  within the space of all possible coefficients. 
We say that a property holds for a generic choice of coefficients, 
if there is a nonempty Zariski open set of the coefficients 
for which this property holds. 
From a probabilistic point of view, ``genericity" implies that 
if $a$ and $\boldomega$ are selected at random, 
then the property holds with \emph{probability one}.
Now, following a standard ``generic smoothness'' argument,
we show that the solution set of such a face subsystem
consists of nonsingular points.

\begin{lemma}
    Let $C_N$ be a cycle graph of $N$ nodes, 
    and let $F$ be a face of $\nabla_{C_N}$.
    For generic choices of $\boldc \in \C^n$ and $a \in \C^*$,
    the complex zero set of the face subsystem $\boldf_F^R$ 
    is either empty or consists of nonsingular isolated points.
\end{lemma}

This statement follows directly from the properties of a generic member
of a linear system of divisors away from the base locus.
Here, we include a short proof for completeness.

\begin{proof}
    The face subsystem $\boldf_F^R$ is a linear combination of
    the system of Laurent polynomials
    \[
        \{ 1 \} \cup 
        \{ x_i x_j^{-1} \}_{(\bolde_i - \bolde_j) \in F}
    \]
    with coefficients that are the images of 
    $\boldc \in \C^n$ and $aR \in \C^n \times \C^n$ 
    under a nonsingular linear transformation.
    Note that the base locus of this system in $(\C^*)^n$ is empty.
    Then by Bertini's theorem, there exists a Zariski open set $U$
    of $\C^n \times \C^n \times \C^n$ such that $(\boldc,aR) \in U$ 
    implies that the zero set of $\boldf_F^R$ in $(\C^*)^n$ 
    is either empty or 0-dimensional and nonsingular.
    
    Since we require $a \in \C^*$ and $R$ nonsingular, by assumption,
    the inverse image of $U$ in $\C^* \times \C^n \times \C^n$
    remains Zariski open.
    Therefore, for generic choices of $\boldc \in \C^n$ and $a \in \C^*$,
    the zero set of the facet subsystem $\boldf_F^R$ consists of nonsingular isolated points.
\end{proof}

By considering the entire polytope $\nabla_{C_N}$ as a face,
the above lemma implies that the generic solution set to
the algebraic Kuramoto equation consists of nonsingular isolated points.

\begin{corollary}\label{cor:f-count}
    For generic choices of $\boldc \in \C^n$ and $a \in \C^*$,
    the complex zero set of $\boldf$ consists of nonsingular isolated points,
    and the total number is a constant which is independent from $\boldc$ and $a$.
\end{corollary}

We establish the root count for $\boldf$ by studying the root count
of individual facet subsystems,
since the complex zeros of $\boldf$
are one-to-one correspondence with complex zeros of
all facet subsystems as show in Ref.~\cite{Chen2019Directed}.
Here, we restate the result in the current context.

\begin{theorem}[Theorem 2,~\cite{Chen2019Directed}]\label{thm:facet-sum}
    For generic choices of $\boldc \in \C^n$ and $a \in \C^*$,
    the complex zeros of the algebraic Kuramoto system $\boldf$ on $C_N$
    with uniform coupling~\eqref{equ:kuramoto-alg}
    are all isolated and nonsingular,
    and the total number is exactly the sum of the complex root count
    of all the facet subsystems $\boldf_F^R$ 
    over all facets $F \in \mathcal{F}(\nabla_{C_N})$.
\end{theorem}

The root counting question (\Cref{prb:one} and \Cref{prb:two})
is now reduced to computing the root count of each facet subsystem.
The root count for facet subsystems associated with a cycle network
with \emph{non-uniform coupling} is established in Ref.~\cite{ChenDavisMehta2018Counting}.
However, under the additional condition of \emph{uniform coupling}
the generic root count could be strictly less.
We analyze this gap using Bernshtein's second theorem, which states that 
the actual $\C^*$-root count for a system of $n$ polynomial in $n$ variables
is strictly less than the BKK bound if and only if there is an initial system
that has a nontrivial solution in $(\C^*)^n$.

\begin{theorem}[D. Bernshtein, \cite{Bernshtein1975Number}, Theorem B]
Let $\mathbf{f} = (f_1,\dots,f_n)^\top$ be a Laurent polynomial system in $n$ complex variables 
with Newton polytopes $P_1,\dots,P_n$. 
If an initial system $\init_{\boldalpha} \mathbf{f}$ has no roots in $(\C^*)^n$ 
for any $\boldalpha \neq \boldzero$, 
then all roots of $\mathbf{f}$ in $(\C^*)^n$ are isolated 
and their number, counting multiplicity, equals the mixed volume of $P_1, \dots, P_n$.
If an initial system $\init_{\boldalpha} \mathbf{f}$ has a root in $(\C^*)^n$ 
for some $\boldalpha \neq \boldzero$, 
then the number of isolated roots of the system $\mathbf{f}$ in $(\C^*)^n$ 
counted according to multiplicity, 
is strictly smaller than the mixed volume of $P_1, \dots, P_n$, 
given this mixed volume is nonzero.
\end{theorem}

The following theorem provides the condition under which 
an initial system we are interest in has a nontrivial solution.

\begin{theorem} \label{prop.div.4}
    Given a facet subsystem $\boldf_F^R$,
    an initial system $\init_{\boldalpha} \boldf_F^R$ for $\boldalpha \ne \boldzero$ 
    has a zero in $(\C^*)^n$ if and only if $N$ is divisible by 4 and 
    $\boldalpha$ is the inner normal vector of 
    the facet $F$ of $\nabla_{C_N}$.
\end{theorem}

\begin{proof}
We shall first consider an initial system induced by a facet $F$ of $\nabla_{C_N}$
listed in \eqref{def_oddfacet} and \eqref{def_evenfacet}.
Let $\boldalpha$ be an inner normal vector of $F$ in $\nabla_{C_N}$,
then the induced initial form is 
\begin{equation}
    \init_{\alpha} \boldf_F^R =
    - a \sum_{(\bolde_i - \bolde_j) \in F} 
    (R_i - R_j) \, (\boldx^{(\bolde_i - \bolde_j)})^\top
\end{equation}
which is equivalent to the system
\begin{equation} \label{equ: unmixed init.system}
    -\frac{1}{a}R^{-1} \operatorname{init}_{\alpha} \boldf_F^R =
    \sum_{(\bolde_i - \bolde_j) \in F} 
    (\bolde_i - \bolde_j) \, (\boldx^{(\bolde_i - \bolde_j)})^\top
    = V \, (\boldx^{V})^\top,
\end{equation}
where $V$ is the facet matrix.

(Odd N) If $N$ is odd, then with the biholomorphic change of variables $\boldy = \boldx^{V}$, the system \eqref{equ: unmixed init.system} is equivalent to 
\begin{equation}
    V^{-1}V \left((\boldy^{-V})^V)\right)^\top=\boldy^\top,
\end{equation}
as far as their zero sets in $(\C^*)^n$ are concerned.
It is easy to see, however, that the system above does not have solutions in $(\C^*)^n$.

(Even N) If $N$ is even,
 then the reduced facet matrix of $F$, 
\begin{equation*}
    V^* = QV = 
    \begin{bmatrix}
        \bolde_1 & \cdots &\bolde_n & \boldh
    \end{bmatrix},
\end{equation*}
has the last column $\boldh$ with $\frac{N}{2}$ entries equal $1$ and $\frac{N}{2}-1$ entries equal $-1$ 
(See proof of \Cref{prop.unimod.equiv} for details). 
Then via the biholomorphic change of variables $\boldx=\boldy^Q$,
the above system is equivalent to
\begin{equation*}
    QV(((\boldy)^Q)^V)^\top = V^* (\boldy^{V^*})^\top.
\end{equation*}
That is, the initial system $\init_{\alpha} \boldf_F^R = \boldzero$
has a $\C^*$-solution if and only if
\begin{equation}\label{equ:init*}
    V^* (\boldy^{V^*})^\top = \boldzero
\end{equation}
has a $\C^*$-solution.
We now show this is possible if and only if $N$ is divisible by 4.

Assume $N$ is divisible by 4,
we shall show~\eqref{equ:init*} has a $\C^*$-solution.
In fact, $\boldh$ is a solution. 
Let $\boldy = \boldh^\top = (h_1,\dots,h_n) \in (\C^*)^n$, then
\begin{align}
    V^*(\boldy^{V^*})^\top = V^* \, ((\boldh^\top)^{V^*})^\top = 
    V^* \, 
    \begin{bmatrix}
        h_1    \\ 
        \vdots \\ 
        h_n    \\ 
        (h_1)^{h_1}\cdots(h_n)^{h_n}
    \end{bmatrix}.
    \label{eq:init h}
\end{align}
Recall that $h_i \in \{+1,-1\}$, so $h_i^{h_i}=h_i$ for each $i=1,\dots,n$. 
Since $N$ is divisible by 4, and since $\boldh$ has exactly $\frac{N}{2}-1$ entries equal $-1$ while the rest are 1's 
the last entry in $(\boldh^\top)^{V^*}$ is 
\[
    (h_1)^{h_1}\cdots(h_n)^{h_n} =
    h_1\cdots h_n=(-1)^{\frac{N}{2}-1}
    = -1.
\] 

So, \eqref{eq:init h} becomes
\begin{align}
    V^* \, (\boldh^\top)^{V^*}= \begin{bmatrix} 
    1 &      &  &h_1\\
      &\ddots&  &\vdots\\
      &      &1 &h_n
    \end{bmatrix} 
    \begin{bmatrix}
        h_1    \\ 
        \vdots \\ 
        h_n    \\ 
        -1
    \end{bmatrix} 
    =
    \begin{bmatrix}
        h_1 - h_1 \\ 
        \vdots    \\ 
        h_n - h_n 
    \end{bmatrix}
    =\boldzero. 
    \label{eq:init 4}
\end{align}
That is, $\boldh \in (\C^*)^n$ is a solution to~\eqref{equ:init*}.
In fact, since~\eqref{equ:init*} is homogeneous, 
$\lambda \boldh \in (\C^*)^n$ for any $\lambda \in \C^*$ will also be solution. 
Moreover, for any other solution $\boldy$, equation \eqref{equ:init*} is equivalent to 
\begin{align}
    \begin{bmatrix} 
        1 &      &     \\
          &\ddots&   \\
          &      &1 
    \end{bmatrix} 
    \begin{bmatrix}
        y_1    \\ 
        \vdots \\ 
        y_n    
    \end{bmatrix} 
    =-(y_1)^{h_1}\cdots(y_n)^{h_n} \boldh. 
\end{align}
Hence, all the solutions of 
\eqref{equ:init*} are of the form form $\lambda \boldh \in (\C^*)^n$,  $\lambda \in \C^*$.
Consequently, the initial system $\init_{\alpha} \boldf_F^R$ has a
1-dimensional zero set in $(\C^*)^n$.

Now assume $N$ is even but not divisible by 4. 
Suppose $\boldx \in (\C^*)^n$ is a zero of the initial system 
$\init_{\boldalpha} \boldf_F^R$, then $\boldy = \boldx^{Q^{-1}} \in (\C^*)^n$ 
is zero of \eqref{equ:init*}, i.e., 
\begin{align}
    V^* \, \boldy^{V^*} = 
    \begin{bmatrix} 
        1 &      &  &h_1    \\
          &\ddots&  &\vdots \\
          &      &1 &h_n
    \end{bmatrix} 
    \begin{bmatrix}
        y_1    \\ 
        \vdots \\ 
        y_n    \\ 
        (y_1)^{h_1}\cdots(y_n)^{h_n}
    \end{bmatrix} 
    =\boldzero. 
    \label{eq:init_not4 }
\end{align}
We have $y_i=-h_i (y_1)^{h_1}\cdots(y_n)^{h_n}$, $i=1,\dots , n$, 
which implies 
\begin{equation}
    y_1^{h_1}\cdots y_n^{h_n}=(-1)^{h_1+\dots +h_n} h_1 \cdots h_n \left((y_1)^{h_1}\cdots(y_n)^{h_n}\right)^{h_1+\dots+ h_n}.
\end{equation}
Recall that $h_1+\dots+ h_n=1$ and $h_1 \cdots h_n=(-1)^{\frac{N}{2}-1}$. 
Hence, the equation above is equivalent to
\begin{equation}
    y_1^{h_1} \cdots y_n^{h_n} = 
    (-1)^{\frac{N}{2}} \cdot y_1^{h_1}\cdots y_n^{h_n} = 
    - y_1^{h_1}\cdots y_n^{h_n}
\end{equation}
since $N$ is not divisible by 4.
This equation implies that $y_1^{h_1} \cdots y_n^{h_n} = 0$,
i.e., $y_k = 0$ for at least one $k \in \{1,\dots,n\}$,
contradicting with the assumption that $\boldy = \boldx^{E^{-1}} \in (\C^*)^n$. 
Therefore we can conclude that 
the initial system $\init_{\boldalpha} \boldf_F^R$ has no zeros in $(\C^*)^n$.

We now show all other initial systems of $\boldf_F^R$
have no zeros in $(\C^*)^n$ for any $N$.
Consider an initial system $\init_{\boldalpha} \boldf_F^R$
induced by a nonzero vector $\boldalpha \in \R^n$ for which 
$(\newt(\boldf_F^R))_{\boldalpha} \ne F$.
If $\init_{\boldalpha} \boldf_F^R$ does not involves the constant term,
then this system is equivalent to
\[
    W (\boldx^W)^\top = \boldzero,
\]
where $W$ consists of less than $N-1$ columns of $V$ when $N$ is odd and less than $N$ columns of $V$ when $N$ is even.
Consequently, $W$ has full column rank.
By the transformation via its Moore-Penrose inverse, we get
\[
    \boldzero = W^+ \boldzero = W^+ W(\boldx^W)^\top = \boldx^W
\]
which implies $\boldx \not\in (\C^*)^n$.
If $\init_{\boldalpha} \boldf_F^R$ involves the constant term,
by Generalized Sard's theorem, its zero set must be 0-dimensional
for generic choices of the constant terms $\boldc \in \C^n$.
But its Newton polytope is of lower dimension,
so by the Bernshtein-Kushnirenko-Khovanskii Theorem \cite{ Bernshtein1975Number, Khovanskii1978Newton, Kushnirenko1976Polyedres},
its zero set in $(\C^*)^n$ must be empty.
\end{proof}

This result shows that for $N$ not divisible by 4,
no initial system of a facet subsystem has a nontrivial $\C^*$-solution.
Therefore, the root count for the facet subsystem agrees with the BKK bound.
Combining with the root count results established in Ref.~\cite{ChenDavisMehta2018Counting},
the root count in cases where $N$ is not divisible by 4 can be derived immediately.

\begin{corollary}
    If $N$ is odd,  a facet subsystem $\boldf_F^R$, 
    with generic choices of $\boldc$ and $a$, 
    has a unique $\C^*$-solution.
\end{corollary}

\begin{corollary}
    If $N$ is even but not divisible by 4, 
    then the number of $\C^*$-solutions to a facet subsystem $\boldf_F^R$,
    with generic choices of $\boldc$ and $a$, 
    is
    $\frac{N}{2}$.
\end{corollary}

For the cases where $N$ is divisible by 4, 
we have identified the unique initial system of a given facet subsystem
that has a nontrivial $\C^*$-solution.
Consequently, the root count, even under the assumption of generic $\boldc$ and $a$,
is strictly less than the BKK bound or the adjacency polytope bound.
We compute the exact root count below.

\begin{corollary} \label{cor: number of subsystem solutions div. by 4}
    If $N$ is divisible by 4, then the number of $\C^*$-solutions 
    to a facet subsystem $\boldf_F^R$ is $\frac{N}{2} - 1$.
\end{corollary}

\begin{proof}
    Let $V$ be the facet matrix associated with $F$,
    and let $V^*$ and $Q$ be the corresponding reduced facet matrix
    and facet reduction matrix respectively.
    Recall that the facet system is given by
    \begin{equation}
        \boldf_F^R =
        \boldc 
        - a \sum_{(\bolde_i - \bolde_j) \in F} 
        (R_i - R_j) \, (\boldx^{(\bolde_i - \bolde_j)})^\top
        =
        \boldc - a R V (\boldx^V)^\top,
    \end{equation}
    where $V$ is the facet matrix associated with $F$.
    By the root counting result established in Ref.~\cite{ChenDavisMehta2018Counting}, Proposition 12, the BKK bound for this system is $\frac{N}{2}$,
    i.e., this system has at most $\frac{N}{2}$ zeros in $(\C^*)^n$,
    and this upper bound is attainable for generic choices coefficients
    (i.e., all coefficients are chosen generically and independently).
    Here we show that due to the special algebraic relations among the coefficients,
    the actual number of zeros in $(\C^*)^n$ is $\frac{N}{2} - 1$.
    
    Via the unimodular change of variables $\boldx = \boldy^{Q}$,
    the above facet system can be transformed into
    \begin{equation}\label{equ:facet-system-y}
        \boldc - a R V (\boldy^{QV})^\top =
        \boldc - a R V (\boldy^{V^*})^\top.
    \end{equation}
    Since the change of variables preserves the number of zeros in $(\C^*)^n$,
    it is thus sufficient to count the number of zeros of this system instead.
    
    All zeros of~\eqref{equ:facet-system-y} in $(\C^*)^n$ are isolated and simple,
    and under this transformation,
    the only initial system with nontrivial $\C^*$-solution is the initial system
    defined by the vector $\boldalpha = (-1,\dots,-1)$.
    Therefore the only zeros outside $(\C^*)^n$ are at infinity, 
    i.e., $\mathbb{CP}^n \setminus \C^n$,
    which we can compute explicitly by considering the homogenization of \eqref{equ:facet-system-y}.
    At infinity, the system is equivalent to 
    \[
        -aRV (\boldy^{V^*})^\top = \boldzero,
    \]
    which is, in turn, is equivalent to
    \[
        V^* (\boldy^{V^*})^\top = \boldzero.
    \]
    This system coincide with the initial system \eqref{equ:init*}
    in the proof of \Cref{prop.div.4},
    which has a unique nonsingular solution in $\mathbb{CP}^n$.
    In other words,~\eqref{equ:facet-system-y} has only one simple zero at infinity.
    Its root count in $(\C^*)^n$ is therefore one less than the BKK bound,
    i.e., the root count is $\frac{N}{2} - 1$.
\end{proof}

\begin{example} (Running example, 4-cycle).
    To illustrate the result of Corollary \ref{cor: number of subsystem solutions div. by 4}, we consider the facet 
    \[
        \conv 
            \left\{
                \left[\begin{smallmatrix} 1\\0\\0 \end{smallmatrix}\right],
                \left[\begin{smallmatrix} \phantom{-}1\\-1\\\phantom{-}0 \end{smallmatrix}\right], 
                \left[\begin{smallmatrix} \phantom{-}0\\-1\\\phantom{-}1 \end{smallmatrix}\right],
                \left[\begin{smallmatrix} 0\\0\\1 \end{smallmatrix} \right]
            \right\}
    \] 
    of the adjacency polytope of a 4-cycle network 
    (the shaded facet in Figure \ref{sec: Adjacency polytope}). 
    With a direct computation, we can verify that the normalized volume of this facet is 2.
    This is the BKK bound of the facet subsystem.
    Therefore, if all coefficients were chosen randomly,
    we expect the facet subsystem to have 2 solutions in $(\C^*)^3$.
    However, due to the special algebraic relations among the coefficients,
    as a result of the uniform-coupling requirement,
    this facet subsystem has only one solution.
    Indeed, with direct symbolic computation,
    we can compute the unique solution which is given by
    \begin{align*}
        x_0 &= 1, &
        x_1 &= \frac{1}{a} \,  \frac{(\omega_1+\omega_2+\omega_3)\omega_1}{\omega_1+\omega_3},\\
        x_2 &= \frac{(\omega_1+\omega_2+\omega_3)\omega_3}{\omega_1+\omega_3}, &
        x_3 &= \frac{1}{a} \,  \frac{(\omega_2+\omega_3)(\omega_1+\omega_2+\omega_3)}{\omega_1+2\omega_2+\omega_3}.
    \end{align*}
    The same argument can be applied to all 6 facet subsystems
    (corresponding to the 6 facets listed in \eqref{facets 4-cycle}.
    We can thus conclude that under the generic uniform-coupling assumption,
    the algebraic Kuramoto equations~\eqref{equ:kuramoto-alg} for 4-cycle graph
    has exactly 6 complex solutions
    even though its BKK bound is 12.
\end{example}

Combining the above corollaries, \Cref{thm:facet-sum},
and the total number facets of $\nabla_{C_N}$,
we get the following generic root count for~\eqref{equ:kuramoto-alg}
under the assumption of generic natural frequency and 
generic but uniform coupling coefficients.

\begin{theorem}
    Given a cycle network of $N$ oscillators with uniform coupling 
    and generically chosen complex constants $a,\omega_i,\dots,\omega_n$,
    the number of isolated complex solutions to the system~\eqref{equ:kuramoto-alg} is
    \[
        \left\{
        \begin{aligned}
            &N \binom{N-1}{ \lfloor (N-1) / 2 \rfloor } && \text{if $N$ is not divisible by 4} \\[2ex]
            &(N-2) \binom{N-1}{N/2-1} && \text{if $N$ is divisible by 4}.
        \end{aligned}
        \right.
    \]
\end{theorem}

These are also the answers to~\Cref{prb:two}, i.e.,
they are the birationally invariant intersection indices 
$[L_1,\dots,L_n]$ for the family of vectors spaces of rational functions
over the toric variety $(\C^*)^n$.
Note that this intersection index is strictly less than the BKK bound
when $N$ is divisible by 4.

\begin{remark}
    In the cases where $N$ is divisible by 4,
    the gap between the birationally invariant intersection index
    and the BKK bound is $2 \binom{N-1}{N/2-1} = \binom{N}{N/2}$, 
    which grows exponentially as $N \to \infty$.
\end{remark}

\section{Concluding remarks}\label{sec: conclusion}

The Kuramoto model is one of the most widely studied models
for 
describing the pervasive phenomenon of spontaneous synchronization in networks of coupled oscillators. 
In this model, frequency synchronization configurations can be formulated as complex solutions to a system of algebraic equations.
Under the assumption of generic natural frequencies and generic non-uniform coupling strength,
the upper bounds to the number of frequency synchronization configurations 
in cycle networks of $N$ oscillators were computed in a recent work~\cite{ChenDavisMehta2018Counting}.
This paper provides a sharper upper bound for the special cases of networks with uniform coupling.
The uniform coupling assumption 
imposes an algebraic condition on the coefficients of the algebraic Kuramoto equations
and 
potentially reduces the maximum number of solutions.
We have established the exact condition under which the maximum root count 
is lower and quantified the gap. 
In particular, if $N$ is not divisible by 4,
then the maximum complex root count of the Kuramoto equations remains the same
with or without the uniform coupling assumption.
On the other hand, if $N$ is divisible by 4,
the uniform coupling assumption significantly lowers the maximum root count.
Indeed, the gap between the bounds is $\binom{N}{N/2}$
which grows exponentially as $N \to \infty$.

\bibliographystyle{abbrv}
\bibliography{library.bib}

\begin{thebibliography}{10}

\bibitem{ARENAS200893}
A.~Arenas, A.~Díaz-Guilera, J.~Kurths, Y.~Moreno, and C.~Zhou.
\newblock Synchronization in complex networks.
\newblock {\em Physics Reports}, 469(3):93 -- 153, 2008.

\bibitem{Baillieul1982}
J.~Baillieul and C.~I. Byrnes.
\newblock {Geometric Critical Point Analysis of Lossless Power System Models}.
\newblock {\em IEEE Transactions on Circuits and Systems}, 29(11):724--737, nov
  1982.

\bibitem{Bernshtein1975Number}
D.~N. Bernshtein.
\newblock {The number of roots of a system of equations}.
\newblock {\em Functional Analysis and its Applications}, 9(3):183--185, 1975.

\bibitem{Bronski2016Graph}
J.~C. Bronski, L.~DeVille, and T.~Ferguson.
\newblock Graph homology and stability of coupled oscillator networks.
\newblock {\em SIAM Journal on Applied Mathematics}, 76(3):1126--1151, 2016.

\bibitem{Chen2019Directed}
T.~Chen.
\newblock {Directed acyclic decomposition of Kuramoto equations}.
\newblock {\em Chaos: An Interdisciplinary Journal of Nonlinear Science},
  29(9):093101, sep 2019.

\bibitem{Chen2019Unmixing}
T.~Chen.
\newblock {Unmixing the Mixed Volume Computation}.
\newblock {\em Discrete and Computational Geometry}, mar 2019.

\bibitem{ChenDavis2018Toric}
T.~Chen and R.~Davis.
\newblock {A toric deformation method for solving Kuramoto equations}.
\newblock oct 2018.
\newblock \url{http://arxiv.org/abs/1810.05690}.

\bibitem{ChenDavisMehta2018Counting}
T.~Chen, R.~Davis, and D.~Mehta.
\newblock {Counting Equilibria of the Kuramoto Model Using Birationally
  Invariant Intersection Index}.
\newblock {\em SIAM Journal on Applied Algebra and Geometry}, 2(4):489--507,
  jan 2018.

\bibitem{dal2019faces}
A.~D'Al\`{\i}, E.~Delucchi, and M.~Micha\l{}ek.
\newblock Many faces of symmetric edge polytopes, 2019.
\newblock \url{https://arxiv.org/abs/1910.05193}.

\bibitem{DelabaysColettaJacquod2016Multistability}
R.~Delabays, T.~Coletta, and P.~Jacquod.
\newblock {Multistability of phase-locking and topological winding numbers in
  locally coupled Kuramoto models on single-loop networks}.
\newblock {\em Journal of Mathematical Physics}, 57(3):032701, mar 2016.

\bibitem{dorfler_synchronization_2014}
F.~D{\"{o}}rfler and F.~Bullo.
\newblock {Synchronization in complex networks of phase oscillators: A survey}.
\newblock {\em Automatica}, 50(6):1539--1564, 2014.

\bibitem{Denes2019Pattern}
K.~Dénes, B.~Sándor, and Z.~Néda.
\newblock Pattern selection in a ring of kuramoto oscillators.
\newblock {\em Communications in Nonlinear Science and Numerical Simulation},
  78:104868, 2019.

\bibitem{Guo1990}
S.~Guo and F.~Salam.
\newblock {Determining the solutions of the load flow of power systems:
  Theoretical results and computer implementation}.
\newblock volume~3, pages 1561--1566, dec 1990.

\bibitem{Ha2012Basin}
S.-Y. Ha and M.-J. Kang.
\newblock On the basin of attractors for the unidirectionally coupled kuramoto
  model in a ring.
\newblock {\em SIAM Journal on Applied Mathematics}, 72(5):1549, 2012.

\bibitem{KavehKhovanskii2010Mixed}
K.~Kaveh and A.~G. Khovanskii.
\newblock {Mixed volume and an extension of theory of divisors}.
\newblock {\em Moscow Mathematical Journal}, 10(2):343--375, 2010.

\bibitem{Khovanskii1978Newton}
A.~G. Khovanskii.
\newblock {Newton polyhedra and the genus of complete intersections}.
\newblock {\em Functional Analysis and Its Applications}, 12(1):38--46, 1978.

\bibitem{Kushnirenko1976Polyedres}
A.~G. Kouchnirenko.
\newblock {Poly{\`{e}}dres de Newton et nombres de Milnor}.
\newblock {\em Inventiones Mathematicae}, 32(1):1--31, 1976.

\bibitem{Kuramoto1975Self}
Y.~Kuramoto.
\newblock {Self-entrainment of a population of coupled non-linear oscillators}.
\newblock Lecture Notes in Physics, pages 420--422. Springer Berlin Heidelberg,
  1975.

\bibitem{ManikTimmeWitthaut2017Cycle}
D.~Manik, M.~Timme, and D.~Witthaut.
\newblock {Cycle flows and multistability in oscillatory networks}.
\newblock {\em Chaos: An Interdisciplinary Journal of Nonlinear Science},
  27(8):083123, aug 2017.

\bibitem{Matsui2011Roots}
T.~Matsui, A.~Higashitani, Y.~Nagazawa, H.~Ohsugi, and T.~Hibi.
\newblock {Roots of Ehrhart polynomials arising from graphs}.
\newblock {\em Journal of Algebraic Combinatorics}, 34(4):721--749, dec 2011.

\bibitem{MolzahnMehtaNiemerg2016Toward}
D.~K. Molzahn, D.~Mehta, and M.~Niemerg.
\newblock {Toward topologically based upper bounds on the number of power flow
  solutions}.
\newblock In {\em Proceedings of the American Control Conference}, volume
  2016-July, pages 5927--5932. IEEE, jul 2016.

\bibitem{Nill2005Classification}
B.~Nill.
\newblock Classification of pseudo-symmetric simplicial reflexive polytopes.
\newblock {\em Algebraic and geometric combinatorics, Contemp. Math. 423},
  pages 269–--282, 12 2006.

\bibitem{RODRIGUES20161}
F.~A. Rodrigues, T.~K.~D. Peron, P.~Ji, and J.~Kurths.
\newblock The kuramoto model in complex networks.
\newblock {\em Physics Reports}, 610:1 -- 98, 2016.

\bibitem{Rogge2004Stability}
J.~A. Rogge and D.~Aeyels.
\newblock Stability of phase locking in a ring of unidirectionally coupled
  oscillators.
\newblock {\em Journal of Physics. A. Mathematical and General}, 37(46):11135,
  2004.

\bibitem{Roy2012Synchronized}
T.~K. Roy and A.~Lahiri.
\newblock Synchronized oscillations on a kuramoto ring and their entrainment
  under periodic driving.
\newblock {\em Chaos, Solitons \& Fractals}, 45(6):888 -- 898, 2012.

\bibitem{SommeseWampler2005Numerical}
A.~J. Sommese and C.~W. Wampler.
\newblock {\em {The Numerical Solution of Systems of Polynomials Arising in
  Engineering and Science}}.
\newblock World Scientific, mar 2005.

\bibitem{Xi2017Synchronization}
K.~Xi, J.~L.~A. Dubbeldam, and H.~X. Lin.
\newblock {Synchronization of cyclic power grids: Equilibria and stability of
  the synchronous state}.
\newblock {\em Chaos: An Interdisciplinary Journal of Nonlinear Science},
  27(1):13109, 2017.

\end{thebibliography}

\end{document}